\documentclass[10pt,a4paper]{amsart}
\usepackage{mathrsfs}
\usepackage{syntonly}
\usepackage{amsmath}
\usepackage{amsthm}
\usepackage{amsfonts}
\usepackage{amssymb}
\usepackage{latexsym}
\usepackage{amscd,amssymb,amsopn,amsmath,amsthm,graphics,amsfonts,mathrsfs,accents,enumerate,verbatim,calc}
\usepackage[dvips]{graphicx}
\usepackage[colorlinks=true,linkcolor=red,citecolor=blue]{hyperref}
\usepackage[all]{xy}
\usepackage{enumitem}
\usepackage{mathtools}
\usepackage{tikz}
\usetikzlibrary{decorations.pathreplacing,decorations.markings}
 \tikzset{
  on each segment/.style={
    decorate,
    decoration={
      show path construction,
      moveto code={},
      lineto code={
        \path [#1]
        (\tikzinputsegmentfirst) -- (\tikzinputsegmentlast);
      },
      curveto code={
        \path [#1] (\tikzinputsegmentfirst)
        .. controls
        (\tikzinputsegmentsupporta) and (\tikzinputsegmentsupportb)
        ..
        (\tikzinputsegmentlast);
      },
      closepath code={
        \path [#1]
        (\tikzinputsegmentfirst) -- (\tikzinputsegmentlast);
      },
    },
  },
  mid arrow/.style={postaction={decorate,decoration={
        markings,
        mark=at position 0.6 with {\arrow[#1]{stealth}} 
      }}},
}
\usetikzlibrary{arrows}
\usetikzlibrary{trees}

\usetikzlibrary{matrix}
\usetikzlibrary{patterns}
\usetikzlibrary{shadings} 
\date{}
\allowdisplaybreaks[4] \footskip=15pt
\renewcommand{\uppercasenonmath}[1]{}

\numberwithin{equation}{section} \theoremstyle{plain}
\newtheorem*{thm*}{Main Theorem}
\newtheorem{thm}{Theorem}[section]
\newtheorem{cor}[thm]{Corollary}
\newtheorem*{cor*}{Corollary}
\newtheorem{lem}[thm]{Lemma}
\newtheorem*{lem*}{Lemma}

\newtheorem*{fact*}{Fact}

\newtheorem*{nota*}{Notation}

\newtheorem*{prop*}{Proposition}
\newtheorem{rem}[thm]{Remark}
\newtheorem*{rem*}{Remark}

\newtheorem*{observation*}{Observation}
\newtheorem{exa}[thm]{Example}
\newtheorem*{exa*}{Example}
\newtheorem{df}[thm]{Definition}
\newtheorem*{df*}{Definition}

\newtheorem*{conj*}{Conjecture}

\newtheorem*{quest*}{Question}

\newtheorem*{ack*}{ACKNOWLEDGEMENTS}

\newcommand{\pf}{\noindent\begin {proof}}
\newcommand{\epf}{\end{proof}}

    \newcommand{\I}{\mathrm{Id}_{M}\oo}

\begin{document}
\begin{center}
{{\bf\large
Three results on extension dimensions of syzygy module categories
}

\vspace{0.5cm}   Pei Luo\footnote{{\color{white}\hspace{1mm}} Corresponding author.}, Zhongkui Liu}

\end{center}

$$\bf  Abstract$$
\leftskip0truemm \rightskip0truemm \noindent
This paper establishes three main results on the extension dimension of syzygy module categories:~(1)~we prove that excellent ring extensions preserve extension dimensions of syzygy module categories,~including syzygy categories of modules of finite projective dimensions;\\
(2)~for cleft extensions, we investigate the behavior of extension dimensions under natural nilpotency and projective conditions;~(3)~for commutative Artin rings, we establish local-global characterisations for the extension dimension.
\leftskip10truemm \rightskip10truemm \noindent
\\[2mm]
{\bf Keywords:} Extension dimension~$\cdot$~Excellent extension~$\cdot$~Cleft extension~$\cdot$~Local-global property.\\
{\bf 2020 Mathematics Subject Classification:} 18G20; 16E10; 13D05.

\leftskip0truemm \rightskip0truemm
\section { \bf Introduction}

The finitistic dimension conjecture~$\mathbf{(FDC)}$ is a longstanding open problem in the representation theory of Artin algebras.~It has been verified for several classes, including monomial algebras~\cite{Green1991}, algebras of the representation dimension at most three~\cite{IT2005}, and Igusa-Todorov algebras~\cite{Wei2009}.

Let $R$ be an Artin algebra,~and $R\text{-}\mathrm{mod}$ denote the category of finitely generated left $R$-modules.~The extension dimension $\mathrm{ext.dim}(R)$ of $R\text{-}\mathrm{mod}$~was introduced by Beligiannis~\cite{4},~and the upper bound of extension dimension (also known as radius)~has been studied by Dao-Takahashi~\cite{DT2014}~under some conditions on rings.~Roughly speaking, the extension dimension measures how many extensions are needed to generate the whole category from a single module.~Note from \cite{4} that $\mathrm{ext.dim}(R)=0$ implies that the representation dimension of $R$ is no more than 2, and thereby the zero extension dimension implies the finitistic dimension conjecture by \cite[Corollary~0.9]{IT2005}.~Recently,~the extension dimension~of syzygy module category, which is a high-level syzygy category,~has been studied in \cite{DT2014,DLT2025,ZTS2025}.~In a significant advance,~Zheng~et~al. \cite{17}~have used the extension dimension to reformulate the finitistic dimension conjecture as follows:
\begin{align}\label{equivelence-FDC}
\mathbf{FDC}~\text{holds~for}~R\Longleftrightarrow
\mathrm{ext.dim}(\Omega^{s}\mathcal{P}^{<\infty}(R))
\leq1~\text{for~some}~s\geq0,
\end{align}
where~$\Omega^{i}(\mathcal{P}^{<\infty}(R))$ denotes  the subcategory of $i$-th syzygies of modules of finite projective dimension.~This equivalent characterization
highlights the importance of understanding how extension dimension behaves under ring homomorphisms.~In this work,~we will take different ring changes toward studying the invariance and finiteness of extension dimension of syzygy module category.

In homological algebra, excellent ring extensions play an important role. Classical homology dimensions, such as global dimension~\cite{Liu1994},~Gorenstein dimension~\cite{11},~and finitistic dimension~\cite{Zhang2018}, satisfy some nice inequalities or equalities under excellent ring extensions.~It should be noted that Zheng et al.~\cite{17} have proved the extension dimension remains invariant under excellent extensions.~In recent years,~representation-finite properties of the syzygy module category~(also called syzygy-finite type) have sparked extensive research in representation theory~\cite{BLM2019},~which is related to the finitistic dimension conjecture. However the representation-finite properties is related to the extension dimension of module category.~Moreover,~it is worth noting that the module category is the $0$-th syzygy module category.~This naturally raises a fundamental question: Does the extension dimension of syzygy module category remain invariant under excellent extension? This paper provides a positive answer to the question, and generalizes the invariance of \cite[Theorem~4.2]{17} to the syzygy module category of any level:

\begin{thm}\label{thm1}
{\rm(Theorem~\ref{key} and~\ref{syzygy-prop-3.8})} Let~$S\geq R$~be an excellent extension of Artin algebras.~Then
\begin{align*}
&\mathrm{ext.dim}(\Omega^{j}(R\text{-}\mathrm{mod}))=
\mathrm{ext.dim}(\Omega^{j}(S\text{-}\mathrm{mod})),\\
&\mathrm{ext.dim}\Omega^{j}(\mathcal{P}^{<\infty}(R))=
\mathrm{ext.dim}\Omega^{j}(\mathcal{P}^{<\infty}(S))
\end{align*}
hold for any~integer~$j\geq0$.
\end{thm}

Consequently, we obtain invariance of the syzygy-finiteness and the finiteness of finitistic dimension under such homomorphisms~(see~Corollary~\ref{prop:3.6}~and~\ref{FDC-cor-3.9}).

On the other hand, the theories of cleft extension, introduced by Beligiannis \cite{Beligiannis2000}, provide a unified framework for a wide range of constructions in representation theory. Indeed, Morita context rings, trivial extensions, tensor rings, and arrow removal algebras can all be naturally realized as cleft extensions. In these structures, many homological dimensions and properties have been investigated.~For example,~Green et al. investigated the behavior of the
finitistic dimension in a cleft extension~\cite{GPS2021};~Kostas studied the generation of derived categories and Gorensteinness in a cleft extension~\cite{Kostas2026,Kostas2025};~Ma et al.~\cite{MZL2026} established some inequalities for the extension dimension of module categories under a cleft coextension.~Note from (\ref{equivelence-FDC})~that the extension dimension of syzygy module categories is related to the finitistic dimension conjecture.~Building on this, we aim to study the invariance for extension dimension of syzygy module categories under a cleft extension.

\begin{thm}\label{thm1.2}
{\rm(Theorem~\ref{thm-clft3.10})}
Let~$R$ and $S$ be Artin algebras, $(S\text{-}\mathrm{mod},R\text{-}\mathrm{mod},\mathrm{i},\mathrm{e},\mathrm{l})$~be a cleft extension,~and let $j$ be a non-negative integer.~If $\mathrm{e}$ preserves projectives,~$\mathrm{l}$ is exact~and~$\mathrm{F}^{n}=0$ for some $n\geq1$,~then
$$\mathrm{ext.dim}(\Omega^{j}(S\text{-}\mathrm{mod}))\leq
\mathrm{ext.dim}(\Omega^{j}(R\text{-}\mathrm{mod}))\leq
n\cdot\mathrm{ext.dim}(\Omega^{j}(S\text{-}\mathrm{mod}))+(n-1).$$
\end{thm}

Beyond ring transformations, local-global principles form another important aspect in commutative algebra.~Many classic homological dimensions, such as projective dimension and global dimension, are known to satisfy the local-global principle.~It is therefore natural to ask whether the extension dimension also admits a local-global characterization.

\begin{thm}\label{thm1.3}
{\rm(Theorem~\ref{th:4.2})}
Let $j$ and $n$ be non-negative integers.~If $R$ is a commutative Artin ring,~then the following statements are equivalent:\\
$(1)$~$\mathrm{ext.dim}(\Omega^{j}(R\text{-}\mathrm{mod}))\leq n$;\\
$(2)$~$\mathrm{ext.dim}(\Omega^{j}(R_{\mathfrak{p}}\text{-}\mathrm{mod}))\leq n$ for every $\mathfrak{p}\in\mathrm{Spec}(R)$;\\
$(3)$~$\mathrm{ext.dim}(\Omega^{j}(R_{\mathfrak{m}}\text{-}\mathrm{mod}))\leq n$ for every $\mathfrak{m}\in\mathrm{Max}(R)$.
\end{thm}

The paper is organized as follows. Section 2 recalls basic notations and definitions. Section 3 proves the invariance of extension dimensions under excellent extensions. Section 4 presents Theorem \ref{thm1.2} and discusses application.~Section 5 establishes the local-global properties of extension dimensions.

\section {\bf Preliminaries}
\subsection{Conventions}
Throughout this paper,
all rings are Artin $k$-algebras with unit, where~$k$ is an arbitrary but fixed commutative Artin ring,~all modules are finitely generated left modules.~Let $R$ be an Artin algebra, $R\text{-}\mathrm{mod}$~denote the category of finitely generated left $R$-modules,~and assume all subcategories of~$R\text{-}\mathrm{mod}$~are full, additive and closed under isomorphisms.~Let~$\mathrm{proj}(R)$~and~$\mathrm{inj}(R)$
denote the subcategory of~finitely generated projective and injective left $R$-modules,~respectively.~The set of non-negative integers is denoted by $\mathbb{N}$.~In particular,~if $R$ is a commutative Artin ring,~then we denote the set of prime ideals~and~maximal ideals of $R$ by~$\mathrm{Spec}(R)$~and~$\mathrm{Max}(R)$,~ respectively.

\subsection{Extension dimension}\label{Section2.2}
Given a subcategory $\mathcal{U}$~of $R\text{-}\mathrm{mod}$,~denote by $\mathrm{add}(\mathcal{U})$ the smallest full subcategory of $R\text{-}\mathrm{mod}$ containing $\mathcal{U}$ and is closed under finite direct sums and direct summands.~Let~$[\mathcal{U}]_{0}$ denote the full subcategory consisting of objects isomorphic to the zero module, and~$[\mathcal{U}]_{1}:=\mathrm{add}(\mathcal{U}).$~For $n\geq2$,~it can be inductively defined as
$$[\mathcal{U}]_{n}:=[\mathcal{U}]_{n-1}\bullet[\mathcal{U}]_{1}
=\mathrm{add}([\mathcal{U}]_{n-1}\ast[\mathcal{U}]_{1}),$$
where $\ast$ denotes the operation of taking extensions between subcategories.
When $\mathcal{U}=\{M\}$ for some $M\in R\text{-}\mathrm{mod}$, we write $[M]_{n}:=[\mathcal{U}]_{n}$ for simplicity.

\begin{df}\label{df:2.1}~{\rm The $\mathbf{extension~dimension}$~(\cite{DT2014}) of $\mathcal{U}$ is defined as
\begin{align*}
\mathrm{ext.dim}(\mathcal{U})&:=\mathrm{inf}\{n\geq0\mid \mathcal{U}\subseteq[N]_{n+1}~\text{for some}~N\in R\text{-}\mathrm{mod} \},
\end{align*}
or $\infty$~if no such integer exists.~In the case~$\mathcal{U}=R\text{-}\mathrm{mod}$,~we simply write~$\mathrm{ext.dim}(R):=\mathrm{ext.dim}(R\text{-}\mathrm{mod})$.}
\end{df}

\begin{exa}\label{exa:2.2}{\rm
(\cite[Example~1.6]{4})\\
$(1)$~$\mathrm{ext.dim}(R)\leq \ell\ell(R)-1$,~where~$\ell\ell(R)$ denotes the Loewy length of $R$.\\
$(2)$~$\mathrm{ext.dim}(R)=0$ if and only if $R$ is representation finite. }
\end{exa}

For a functor $H:R\text{-}\mathrm{mod}\rightarrow S\text{-}\mathrm{mod}$ and a subcategory $\mathcal{X}$ of $R\text{-}\mathrm{mod}$, define $H(\mathcal{X}):=\{H(X)\mid X\in\mathcal{X}\}$.

\begin{lem}\label{lem:2.4}
{\rm(\cite[Lemma~2.4]{17})} Let $H:R\text{-}\mathrm{mod}\rightarrow S\text{-}\mathrm{mod}$ be an exact functor,~and let~$\mathcal{X}$ be a subcategory of $R\text{-}\mathrm{mod}$.~Then
$H([\mathcal{X}]_{n})
\subseteq [H(\mathcal{X})]_{n}$~holds for any~$n\in\mathbb{N}$.
\end{lem}

\subsection{(Co-)Syzygy module}
Let $n$ be a positive integer,~$M\in R\text{-}\mathrm{mod}$. Recall from \cite[Definition~2.15]{ZTS2025} that an $R$-module~$K$ is called $n$-th syzygy module of $M$, if there is an exact sequence
$$\cdots\rightarrow P_{n}\rightarrow P_{n-1}\rightarrow\cdots\rightarrow P_{1}\rightarrow P_{0}\rightarrow M\rightarrow0$$
in $R\text{-}\mathrm{mod}$ such that each $P_{i}$ is projective for $i\geq0$, and~$K\cong\mathrm{Im}(P_{n}\rightarrow P_{n-1})$. The $n$-th cosyzygy module can be defined dually.~In this paper,~we denote~$\Omega^{n}_{R}(M)$ and $\Omega^{-n}_{R}(M)$ the $n$-th syzygy module and $n$-th cosyzygy module of $M$, respectively.~By Schanuel's Lemma,~the (co-)~syzygy modules are
independent of the choice of the projective~(resp.~injective) resolution of $M$ up to projective (resp.~injective)~summands.~Moreover,~we~denote by $\Omega^{n}(R\text{-}\mathrm{mod})$ the full subcategory of $R\text{-}\mathrm{mod}$ consisting of all $n$-syzygy~modules. Clearly, $\Omega^{0}(R\text{-}\mathrm{mod})=R\text{-}\mathrm{mod}$.~For a subcategory $\mathcal{X}$ of $R\text{-}\mathrm{mod}$,~set~$\Omega^{n}(\mathcal{X}):=\{\Omega^{n}(X)\mid X\in\mathcal{X}\}$.

\subsection{restriction of scalars}\label{Section2.5}

Let $f:R\rightarrow S$ be a homomorphism of rings.~Every left~$S$-module $M$ acquires a left~$R$-module structure via restriction of scalars:~for~$m\in M$ and $r\in
R$,~set~$r\cdot m:=f(r)\cdot m$.~Since $S$ is a finitely generated~left $R$-module, the corresponding restriction functor transforms all finitely generated left $S$-modules into finitely generated left $R$-modules,~which is denoted by
$f_{*}:S\text{-}\mathrm{mod}\rightarrow R\text{-}\mathrm{mod}$.~By \cite[Example 1.1.47]{21}, every exact sequence in $S\text{-}\mathrm{mod}$ remains exact in $R\text{-}\mathrm{mod}$ after applying $f_{*}$.

\section{Extension dimensions under excellent extensions} 
Let $R$ be a subring of a ring $S$ such that $R$ and $S$ have the same unit. Then $S$ is called a ring extension of $R$ and denoted by $S\geq R$.

\begin{df}{\rm(\cite[Definition~2.1]{11})
A ring extension $S\geq R$ is called an \textbf{excellent~extension}~if the following conditions hold:\\
$(1)$~For a submodule $_{S}N$ of $_{S}M$, if $_{R}N$ is a direct summand of $_{R}M$,~then $_{S}N$ is a direct summand of $_{S}M$;\\
$(2)$~There exists a finite set~$\{s_{1},\cdots,s_{n}\}$ in $S$ such that~$S=\sum^{n}_{j=1}
Rs_{j}$;\\
$(3)$~$_{R}S$ and $S_{R}$ are free with a common basis~$\{s_{1},\cdots,s_{n}\}$ such that~$Rs_{j}=s_{j}R$ for $1\leq j\leq n$.}
\end{df}

\begin{lem}\label{lem3.1}
Let~$f:R\rightarrow S$~be a ring homomorphism~such that~$S\otimes_{R}-$ is an exact functor,~and let~$n\in\mathbb{N}$.~Then\\
$(1)$~$S\otimes_{R}\Omega^{n}(\mathcal{X})\subseteq\Omega^{n}(S\otimes_{R}\mathcal{X})$ holds for any subcategory~$\mathcal{X}$ of $R\text{-}\mathrm{mod}$;\\
$(2)$~If $_{R}S$ is a projective module, then $_{R}Y\in\Omega^{n}(R\text{-}\mathrm{mod})$ for any $Y\in\Omega^{n}(S\text{-}\mathrm{mod})$.
\end{lem}
\begin{proof}
(1)~Let~$M\in\Omega^{n}(\mathcal{X})$,~i.e.~$M$ is an~$n$-th syzygy object of some module $X\in\mathcal{X}$.~Since~$S\otimes_{R}-$ is an exact functor and preserves projectives,~$S\otimes_{R}M$ is an $n$-th syzygy object of $S\otimes_{R}X$ by~\cite[Lemma~3.9]{MZL2026}.~Hence~$S\otimes_{R}M\in\Omega^{n}(S\otimes_{R}\mathcal{X})$.

(2)~First,~for any~$Q\in\mathrm{proj}(S)$,~there exists $Q^{'}\in S\text{-}\mathrm{mod}$ such that~$Q\oplus Q^{'}\cong S^{(I)}$ as $S$-modules, and it is an isomorphism as $R$-modules. Since~$_{R}S$ is projective,~we have that
$_{R}Q$ is projective.~Next, let $M\in S\text{-}\mathrm{mod}$ such that~$Y$ is an $n$-th syzygy object of $M$, then there is an exact sequence in $S\text{-}\mathrm{mod}$
$$0\rightarrow Y\rightarrow Q_{n-1}\rightarrow\cdots\rightarrow Q_{1}\rightarrow Q_{0}\rightarrow M\rightarrow0$$
with each $Q_{i}\in\mathrm{proj}(S)$ for $0\leq i\leq n-1$.~This exact sequence is also exact in $R\text{-}\mathrm{mod}$.~Since~$_{R}(Q_{i})\in\mathrm{proj}(R)$,~we get that~$_{R}(Y)$ is a syzygy object of $_{R}M$.~Hence~$_{R}(Y)\in\Omega^{n}(R\text{-}\mathrm{mod})$.
\end{proof}

Our first main result is now presented as below.

\begin{thm}\label{key}
Let~$S\geq R$~be an excellent ring extension.~Then
$$\mathrm{ext.dim}(\Omega^{j}(R\text{-}\mathrm{mod}))=
\mathrm{ext.dim}(\Omega^{j}(S\text{-}\mathrm{mod}))$$
holds for any~$j\in\mathbb{N}$.
\end{thm}
\begin{proof} Assume $n:=\mathrm{ext.dim}(\Omega^{j}(S\text{-}\mathrm{mod}))$.~Then there exists~$Y\in S\text{-}\mathrm{mod}$ such that~$\Omega^{j}(S\text{-}\mathrm{mod})
\subseteq[Y]_{n+1}$.~Let~$X\in\Omega^{j}(R\text{-}\mathrm{mod})$.~Then, by Lemma~\ref{lem3.1}, we have
$$S\otimes_{R}X\in S\otimes_{R}\Omega^{j}(R\text{-}\mathrm{mod})
\subseteq\Omega^{j}(S\text{-}\mathrm{mod})\subseteq[Y]_{n+1}.$$
Therefore,~there exists~an exact sequence of $S$-modules
$$0\rightarrow M_{1}\rightarrow (S\otimes_{R}X)\oplus X^{'}\rightarrow M_{2}\rightarrow0$$
such that~$M_{1}\in[Y]_{1}$ and~$M_{2}\in[Y]_{n}$.~Note that the sequence is also exact in $R\text{-}\mathrm{mod}$,~thus we can prove that $_{R}(S\otimes_{R}X)\in[_{R}Y]_{n+1}$ by induction on $n$.~Moreover,~the freeness of $S$ as $R$-module implies $_{R}X$ is a direct summand of $_{R}(S\otimes_{R}X)$.Thus,~$X\in[_{R}Y]_{n+1}$ and $\mathrm{ext.dim}(\Omega^{j}(R\text{-}\mathrm{mod}))\leq n$.

Assume $m:=\mathrm{ext.dim}(\Omega^{j}(R\text{-}\mathrm{mod}))$.~Then there exists~$T\in R\text{-}\mathrm{mod}$ such that~$\Omega^{j}(R\text{-}\mathrm{mod})
\subseteq[T]_{m+1}$.~Let $Y\in\Omega^{j}(S\text{-}\mathrm{mod})$.~Then $_{R}Y\in
\Omega^{j}(R\text{-}\mathrm{mod})$ by Lemma \ref{lem3.1},~and~$S\otimes_{R}Y
\in[S\otimes_{R}T]_{m+1}$ by Lemma~\ref{lem:2.4}.~Moreover,~by \cite[Lemma~1.1]{Xue1996}, we have $_{S}Y$ is a direct summand of $S\otimes_{R}Y$.~Thus~$Y\in[S\otimes_{R}T]_{m+1}$ and~$\mathrm{ext.dim}(\Omega^{j}(S\text{-}\mathrm{mod}))\leq m$.
\end{proof}

\begin{rem}\label{rem-key}
{\rm Indeed, if set $j=0$ in Theorem \ref{key}, then we obtain the invariance of extension dimension in \cite[Theorem~4.2]{17}, since $\Omega^{0}(S\text{-}\mathrm{mod})
=S\text{-}\mathrm{mod}$ for any Artin ring $S$.}
\end{rem}

We will apply Theorem~\ref{key} in the following example.

\begin{exa}~{\rm Let $R$ be a finite-dimensional algebra over a field $k$.~It follows from \cite[Example~2.2]{11} that there are following excellent extensions:\\
$(1)$~The matrix ring $M_{n}(R)$ is an excellent extension~of $R$;\\
$(2)$~The skew group ring $R\ast G$ is an excellent extension of $R$, where $G$ is a finite group~with $\mid G\mid^{-1}\in R$;\\
$(3)$~$R\otimes_{k}F$ is an excellent extension of $R$,~where $F$ be a finite separable field extension of $k$;\\
$(4)$~$kA$ is an excellent extension of $kB$,~where $A$ is a finite group,~$B$ is a normal subgroup of $A$, and $B$ contains a Sylow subgroup of $A$.

Therefore,~combining Theorem~\ref{key}, the following equalities hold for $j\geq0$:\\
$(1^{'})$~$\mathrm{ext.dim}(\Omega^{j}(R\text{-}\mathrm{mod}))=
\mathrm{ext.dim}(\Omega^{j}(M_{n}(R)\text{-}\mathrm{mod}))$;\\
$(2^{'})$~$\mathrm{ext.dim}(\Omega^{j}(R\text{-}\mathrm{mod}))=
\mathrm{ext.dim}(\Omega^{j}((R\ast G)\text{-}\mathrm{mod}))$;\\
$(3^{'})$~$\mathrm{ext.dim}(\Omega^{j}(R\text{-}\mathrm{mod}))=
\mathrm{ext.dim}(\Omega^{j}((R\otimes_{k}F)\text{-}\mathrm{mod}))$;\\
$(4^{'})$~$\mathrm{ext.dim}(\Omega^{j}(kA\text{-}\mathrm{mod}))=
\mathrm{ext.dim}(\Omega^{j}(kB\text{-}\mathrm{mod}))$.}
\end{exa}

Recently, syzygy finite algebras~\cite{BLM2019} form a crucial class for which the finitistic dimension conjecture is known to hold.~Let $n\in\mathbb{N}$.~Recall that an algebra $R$ is $n$-$\mathbf{syzygy~finite}$~if the number of non-isomorphic indecomposable objects in~$\Omega^{n}(R\text{-}\mathrm{mod})$~is finite.~In particular,~the case $n=0$~recovers the classical representation-finite algebras.~Next,~we generalize the invariance of representation-finite type of algebras in~\cite[Theorem~1.1]{11} to the syzygy-finite type.

\begin{cor}\label{prop:3.6}
Let~$S\geq R$~be an excellent ring extension.~Then $R$~is $n$-syzygy finite if and only if $S$~is $n$-syzygy finite.
\end{cor}
\begin{proof} Obviously, $R$~is $n$-syzygy finite if and only if~$\mathrm{ext.dim}(\Omega^{n}(R\text{-}\mathrm{mod}))=0$.~Therefore,~the statement follows by Theorem~\ref{key}.
\end{proof}

Let $\mathrm{pd}_{R}(M)$ denote the projective dimension of an $R$-module $M$,~and set $\mathcal{P}^{<\infty}(R):=\{M\in R\text{-}\mathrm{mod}\mid \mathrm{pd}_{R}(M)<\infty\}$.~The finitistic dimension of~$R$, written as $\mathrm{fin.dim}(R)$, is defined as $\mathrm{sup}\{\mathrm{pd}_{R}(M)\mid M\in\mathcal{P}^{<\infty}(R)\}$.~In the reset of this section,~we investigate the finiteness of finitistic dimension under an excellent ring extension.~We begin with
the following lemma.

\begin{lem}\label{syzygy-lemme-3.7}
Let~$S\geq R$~be an excellent ring extension,~and~let~$j\in\mathbb{N}$.~Then
for any~$Y\in\Omega^{j}(\mathcal{P}^{<\infty}(S))$, one has $_{R}Y\in\Omega^{j}(\mathcal{P}^{<\infty}(R))$.
\end{lem}
\begin{proof}
Let $Y\in\Omega^{j}(\mathcal{P}^{<\infty}(S))$.~Then there exists $U\in
\mathcal{P}^{<\infty}(S)$ such that $Y$ is $j$-th syzygy module of $U$.~Hence, there is an exact sequence in $S\text{-}\mathrm{mod}$
$$0\rightarrow Y\rightarrow Q_{j-1}\rightarrow Q_{j-2}\rightarrow\cdots\rightarrow
Q_{0}\rightarrow U\rightarrow0,$$
such that $Q_{i}\in\mathrm{proj}(S)$ for $0\leq i\leq j-1$.~This sequence is also exact in~$R\text{-}\mathrm{mod}$. Moreover,~$_{R}S$ is free implies that $_{R}Q_{i}\in\mathrm{proj}(R)$.~Thus $_{R}Y$ is a $j$-th syzygy object of $_{R}U$. On the other hand,~$_{R}U\in\mathcal{P}^{<\infty}(R)$, hence~$_{R}Y\in
\Omega^{j}(\mathcal{P}^{<\infty}(R))$
\end{proof}

\begin{thm}\label{syzygy-prop-3.8}
Let~$S\geq R$~be an excellent ring extension.~Then
$$\mathrm{ext.dim}(\Omega^{i}(\mathcal{P}^{<\infty}(R)))=
\mathrm{ext.dim}(\Omega^{i}(\mathcal{P}^{<\infty}(S))),~\text{for~all~}i\geq0.$$
\end{thm}
\begin{proof}
Combining Lemma~\ref{syzygy-lemme-3.7}, we can easily prove this Theorem by a proof process similar to Theorem \ref{key}.
\end{proof}

By Theorem~\ref{syzygy-prop-3.8}~and the equivalence~(\ref{equivelence-FDC}) in Sect. 1,~we have the following equivalent characterization for finitistic dimension conjecture.

\begin{cor}\label{FDC-cor-3.9}
{\rm(\cite[Theorem~3.2]{Zhang2018})}
Let~$S\geq R$~be an excellent extension.~Then~$R$ satisfies~$\mathbf{FDC}$ if and only if so does~$S$.
\end{cor}

\section{Extension dimensions under cleft extensions}
In this section, we devote to studying the extension dimension of syzygy module category under a cleft extension.~Recall from \cite[Definition~2.1]{Beligiannis2000} that a $\mathbf{cleft~extension}$ of $S\text{-}\mathrm{mod}$ is a module category $R\text{-}\mathrm{mod}$ with the functors
$$\xymatrix@R=1.8cm@C=1.8cm{
S\text{-}\mathrm{mod}
    \ar^-{\mathrm{i}}[r]
& R\text{-}\mathrm{mod}
    \ar^-{\mathrm{e}}[r]
& S\text{-}\mathrm{mod}
    \ar_-{\mathrm{l}}@/_1.8pc/[l]},$$
and satisfying the following:\\
(1)~The functor $\mathrm{e}$ is faithful and exact;\\
(2)~The pair $(\mathrm{l},\mathrm{e})$ is an adjoint pair;\\
(3)~There is a natural isomorphism $\eta:\mathrm{ei}\rightarrow \mathrm{Id}_{S\text{-}\mathrm{mod}}$.

In this paper,~we denote $(S\text{-}\mathrm{mod},R\text{-}\mathrm{mod},\mathrm{i},\mathrm{e},\mathrm{l})$ as the cleft extension of $S\text{-}\mathrm{mod}$.~In \cite{Kostas2026},~the notion of endofunctors on the cleft extension $(S\text{-}\mathrm{mod},R\text{-}\mathrm{mod},\mathrm{i},\mathrm{e},
\mathrm{l})$ has been introduced.~Denote by $\mu:\mathrm{le}\rightarrow\mathrm{Id}_{R\text{-}\mathrm{mod}}$ the counit of the adjoint pair $(\mathrm{l},\mathrm{e})$.~Notice that $\mu_{X}:\mathrm{le}(X)\rightarrow X$ is an epimorphism for any $X\in R\text{-}\mathrm{mod}$.~Then the following sequence
\begin{align}\label{exact-seq-endofunctor}
0\rightarrow \mathrm{G}(X)\rightarrow \mathrm{le}(X)\rightarrow X\rightarrow0
\end{align}
is exact in~$R\text{-}\mathrm{mod}$, where $\mathrm{G}(X):=\mathrm{Ker}(\mu_{X})$. In particular,~$\mathrm{i}(N)\in R\text{-}\mathrm{mod}$ for any $N\in S\text{-}\mathrm{mod}$.~And then define $\mathrm{F}(N):=\mathrm{e}(\mathrm{G}(\mathrm{i}(N)))$.~Therefore,~the assignment~$N\mapsto \mathrm{F}(N)$ defines an endofunctor $\mathrm{F}:S\text{-}\mathrm{mod}\rightarrow S\text{-}\mathrm{mod}$.~Note that there ie an isomorphism~$\mathrm{el}\cong \mathrm{Id}_{S\text{-}\mathrm{mod}}\oplus \mathrm{F}$.

\begin{lem}\label{lem-endofunctor2.10}
{\rm(\cite[Lemma~2.2]{MZL2026})}~Let $(S\text{-}\mathrm{mod},R\text{-}\mathrm{mod},\mathrm{i},\mathrm{e},
\mathrm{l})$ be a cleft extension. The following statements hold for $n\geq1$:\\
$(1)$~$\mathrm{F}^{n}=0$ if and only if $\mathrm{G}^{n}=0$;\\
$(2)$~For any $M\in R\text{-}\mathrm{mod}$, there is an exact sequence
$$0\rightarrow \mathrm{G}^{n}(M)\rightarrow \mathrm{l}\mathrm{F}^{n-1}\mathrm{e}(M)\rightarrow \mathrm{G}^{n-1}(M)\rightarrow0.$$
\end{lem}

\begin{lem}\label{lem-Cosyzygy3.11}
{\rm(\cite[Lemma~3.5]{ZH2022})}
Let~$0\rightarrow X_{n}\rightarrow \cdots\rightarrow X_{1}\rightarrow
X_{0}\rightarrow M\rightarrow0$~be an exact sequence in $R\text{-}\mathrm{mod}$ with $n\geq0$.~Then
$$M\in[X_{0}]_{1}\bullet[\Omega^{-1}(X_{1})]_{1}\bullet\cdots\bullet
[\Omega^{-n}(X_{n})]_{1}.$$
\end{lem}

Similar to \cite[Lemma~4.7,~4.9]{MZL2026},~we obtain a dual result as below.

\begin{lem}\label{lem-clft4.9}
Let~$X,Y\in R\text{-}\mathrm{mod}$~with $X\in[Y]_{n}$ for some integer $n\geq1$.~Then there exist $I\in\mathrm{inj}(R)$ and $j$-th syzygy $\Omega^{-j}(Y)$ of $Y$ for some $j\geq0$ such that $[\Omega^{-j}(X)]_{1}\subseteq[\Omega^{-j}(Y)\oplus I]_{n}$.
\end{lem}

\begin{thm}\label{thm-clft3.10}
Let~$(S\text{-}\mathrm{mod},R\text{-}\mathrm{mod},\mathrm{i},\mathrm{e},\mathrm{l})$~be a cleft extension.~If $\mathrm{e}$ preserves projectives,~$\mathrm{l}$ is exact~and~$\mathrm{F}^{n}=0$ for some $n\geq1$,~then, for any $j\geq0$, one has
$$\mathrm{ext.dim}(\Omega^{j}(S\text{-}\mathrm{mod}))\leq
\mathrm{ext.dim}(\Omega^{j}(R\text{-}\mathrm{mod}))\leq
n\cdot\mathrm{ext.dim}(\Omega^{j}(S\text{-}\mathrm{mod}))+(n-1).$$
\end{thm}
\begin{proof}
We first show the right-hand side of the inequality.~If~$\mathrm{ext.dim}(\Omega^{j}(S\text{-}\mathrm{mod}))=\infty$,~then inequality holds trivially.~Assume $a:=
\mathrm{ext.dim}(\Omega^{j}(S\text{-}\mathrm{mod}))<\infty$.~Then there exists $N\in
S\text{-}\mathrm{mod}$ such that $\Omega^{j}(S\text{-}\mathrm{mod})\subseteq
[N]_{a+1}$.~Let~$X\in\Omega^{j}(R\text{-}\mathrm{mod})$.~Since~$\mathrm{F}^{n}=0$, by Lemma~\ref{lem-endofunctor2.10},~we obtain the following exact sequence in $R\text{-}\mathrm{mod}$:
\begin{align}\label{equality-endofunctor}
0\rightarrow\mathrm{l}\mathrm{F}^{n-1}\mathrm{e}(X)\rightarrow
\mathrm{l}\mathrm{F}^{n-2}\mathrm{e}(X)\rightarrow\cdots\rightarrow
\mathrm{lFe}(X)\rightarrow \mathrm{le}(X)\rightarrow X\rightarrow0.
\end{align}
Note that $\mathrm{el}\cong\mathrm{Id}_{S\text{-}\mathrm{mod}}\oplus \mathrm{F}$ and functors $\mathrm{l}$ and $\mathrm{e}$ preserve projectives,~this implies $\mathrm{F}$ preserves projectives.~Therefore,~by Lemma~\cite[Lemma~3.9]{MZL2026},~we obtain
$$\mathrm{F}^{t}\mathrm{e}(X)\in\Omega^{j}(S\text{-}\mathrm{mod})
\subseteq[N]_{a+1},~\forall~1\leq t\leq n-1.$$
Since $\mathrm{l}$ is an exact functor,~$\mathrm{l}\mathrm{F}^{t}\mathrm{e}(X)\in
[\mathrm{l}(N)]_{a+1}$ by~Lemma~\ref{lem:2.4}.~Combining (\ref{equality-endofunctor}) and Lemma~\ref{lem-Cosyzygy3.11},~we obtain that
\begin{align*}
X&\in[\mathrm{le}(X)]_{1}\bullet[\Omega^{-1}(\mathrm{l}\mathrm{F}\mathrm{e}(X))]_{1}
\bullet\cdots\bullet
[\Omega^{-(n-1)}(\mathrm{l}\mathrm{F}^{n-1}\mathrm{e}(X))]_{1}\\
&\subseteq[\mathrm{l}(N)]_{a+1}\bullet[\Omega^{-1}(\mathrm{l}(N))\oplus I_{1}]_{a+1}
\bullet\cdots\bullet
[\Omega^{-(n-1)}(\mathrm{l}(N))\oplus I_{n-1}]_{a+1}~\text{~~~~~(by~Lemma~\ref{lem-clft4.9})}\\
&\subseteq[\oplus^{n-1}_{t=0}\Omega^{-t}(\mathrm{l}(N))\oplus(I_{1}\oplus\cdots\oplus
I_{n-1})]_{n(a+1)}\\
&\subseteq[\oplus^{n-1}_{t=0}\Omega^{-t}(\mathrm{l}(N))\oplus\mathrm{D}(R)]_{n(a+1)},
\end{align*}
where $\mathrm{D}(-):=\mathrm{Hom}_{k}(-,E(k/\mathrm{rad}(k)))$ denotes the usual duality and satisfies $\mathrm{inj}(R)=\mathrm{add}(\mathrm{D}(R))$, $E(k/\mathrm{rad}(k))$ denotes the injective envelope of $k/\mathrm{rad}(k)$.~Hence,~$\mathrm{ext.dim}
(\Omega^{j}(R\text{-}\mathrm{mod}))\leq n(a+1)-1$.

On the other hand, we also may assume that~$b:=\mathrm{ext.dim}
(\Omega^{j}(R\text{-}\mathrm{mod}))<\infty$.~Then there exists $M\in R\text{-}\mathrm{mod}$ such that~$\Omega^{j}(R\text{-}\mathrm{mod})\subseteq[M]_{b+1}$.~Let~$Y
\in\Omega^{j}(S\text{-}\mathrm{mod})$.~Since~$\mathrm{l}$ and $\mathrm{e}$ are exact functors and~$\mathrm{l}$ preserves projectives,~by \cite[Lemma~3.9]{MZL2026},~we have
$\mathrm{el}(Y)\in[\mathrm{e}(M)]_{b+1}$.~Finally,~it follows by $\mathrm{el}\cong\mathrm{Id}_{S\text{-}\mathrm{mod}}\oplus \mathrm{F}$ that~$Y$ is a direct summand of $\mathrm{el}(Y)$,~and $Y\in[\mathrm{e}(M)]_{b+1}$.~Consequently,~$\mathrm{ext.dim}
(\Omega^{j}(S\text{-}\mathrm{mod}))\leq\mathrm{ext.dim}
(\Omega^{j}(R\text{-}\mathrm{mod}))$.
\end{proof}

\begin{rem}\label{rem-Cleft-syzygy-3.15}
{\rm Let $(S\text{-}\mathrm{mod},R\text{-}\mathrm{mod},\mathrm{i},\mathrm{e},\mathrm{l})$ be a cleft extension.~The following statements hold:\\
$(1)$~Note that $\mathrm{e}(M)$ and $\mathrm{F}(M)$ belong to $\Omega^{j}(\mathcal{P}^{<\infty}(S))$ for any $M\in\Omega^{j}(\mathcal{P}^{<\infty}(R))$  and any $j\geq0$ under the condition $\mathrm{e}$ preserves projectives.~Hence,
if $\mathrm{e}$ preserves projectives,~$\mathrm{l}$ is exact~and~$\mathrm{F}$ is nilpotent,~then, similar to the proof process of Theorem~\ref{thm-clft3.10}, one has
$$\mathrm{ext.dim}(\Omega^{j}(\mathcal{P}^{<\infty}(S)))\leq
\mathrm{ext.dim}(\Omega^{j}(\mathcal{P}^{<\infty}(R)))\leq
n\cdot\mathrm{ext.dim}(\Omega^{j}(\mathcal{P}^{<\infty}(S)))+(n-1).$$
In particular,~if~$\mathrm{fin.dim}(S)<\infty$, then~$\mathrm{ext.dim}(\Omega^{j}(\mathcal{P}^{<\infty}(R)))\leq
2n-1$~by~(\ref{equivelence-FDC}) in Sect.1.\\
(2)~If the induced endofunctor $\mathrm{F}=0$, then $\mathrm{el}\cong\mathrm{Id}_{S\text{-}\mathrm{mod}}$ by the isomorphism $\mathrm{el}\cong \mathrm{F}\oplus\mathrm{Id}_{S\text{-}\mathrm{mod}}$.~Moreover, by Lemma~\ref{lem-endofunctor2.10},~we have $\mathrm{G}=0$ and thereby $\mathrm{le}\cong
\mathrm{Id}_{R\text{-}\mathrm{mod}}$ by (\ref{exact-seq-endofunctor}).~Consequently,~$\mathrm{e}:R\text{-}\mathrm{mod}
\rightarrow S\text{-}\mathrm{mod}$ is an equivalent functor,~and~by~\cite[Theorem~5.4]{ZTS2025},~we get
$$\mathrm{ext.dim}(\Omega^{j}(R\text{-}\mathrm{mod}))=
\mathrm{ext.dim}(\Omega^{j}(S\text{-}\mathrm{mod})),~\forall~j\geq0.$$
}\end{rem}

\begin{lem}\label{lem:syzygy-equivalence}
Let $R=\prod^{2}_{i=1}R_{i}$ be product of rings.~Then there is an equivalence of categories for $j\geq0$:
\[
\Omega^{j}(R\text{-}\mathrm{mod})\simeq\Omega^{j}(R_{1}\text{-}\mathrm{mod})\times \Omega^{j}(R_{2}\text{-}\mathrm{mod}).
\]
\end{lem}
\begin{proof}
Let~$e_{1}=(1,0)$ and~$e_{2}=(0,1)$ be the central idempotents of $R$ and $e_{1}+e_{2}=1_{R}$.~Then $X\cong
e_{1}X\oplus e_{2}X$ for any $X\in R\text{-}\mathrm{mod}$,~where each $e_{i}X$ carries a natural structure of a left $R_{i}$-module via the projection $R\rightarrow R_{i}$.~By \cite[Corollary~1.3.17]{BK2000}, there is an equivalence
\begin{align*}
\Phi\colon R\text{-}\mathrm{mod}\longrightarrow R_1\text{-}\mathrm{mod} \times R_2\text{-}\mathrm{mod},\quad X\mapsto (e_{1}X,e_{2}X),
\end{align*}
with quasi-inverse $\Psi\colon R_1\text{-}\mathrm{mod} \times R_2\text{-}\mathrm{mod}\longrightarrow
R\text{-}\mathrm{mod}$~satisfying $ G(Y_{1},Y_{2})=Y_{1}\oplus Y_{2}$,
where each $Y_{i}$ is regarded as an $R$-module via the projection $R\rightarrow R_{i}$. Indeed, $P_1\oplus P_2$ is a projective $R$-module for any $P_1\in\mathrm{proj}(R_1)$ and $P_2\in\mathrm{proj}(R_2)$,~and $e_{i}P$ is a projective $R_i$-module for any projective $R$-module $P$ and $1\leq i\leq 2$.~Hence,~it follows from~\cite[Lemma~3.9]{MZL2026} that the restricted functors
\begin{align*}
&\Phi\mid_{\Omega^{j}(R\text{-}\mathrm{mod})}:\Omega^{j}(R\text{-}\mathrm{mod})\rightarrow
\Omega^{j}(R_{1}\text{-}\mathrm{mod}) \times \Omega^{j}(R_{2}\text{-}\mathrm{mod}),\\
&\Psi\mid_{\Omega^{j}(R_1\text{-}\mathrm{mod})\times\Omega^{j}(R_2\text{-}\mathrm{mod})}:
\Omega^{j}(R_{1}\text{-}\mathrm{mod})\times\Omega^{j}(R_{2}\text{-}\mathrm{mod})
\rightarrow\Omega^{j}(R\text{-}\mathrm{mod})
\end{align*}
are well-defined. Since $\Phi$ and $\Psi$ are quasi-inverse on the whole category, their restrictions remain quasi-inverse on these full subcategories. Therefore the restricted $\Phi$ is an equivalence of categories.
\end{proof}

Let $\mathcal{X}_{1},\cdots,\mathcal{X}_{n}$ be subcategories of $R\text{-}\mathrm{mod}$.~By \cite[Sect.~1.1.11]{BK2000}, we form the product category $\mathcal{X}_{1}\times\cdots\times\mathcal{X}_{n}$ as follows. The objects of $\mathcal{X}_{1}\times\cdots\times\mathcal{X}_{n}$ are formed as $(X_{1},\cdots,X_{n})$,~where $X_{i}\in\mathcal{X}_{i}$ for $1\leq i\leq n$, and a morphism is formed as
$(f_{1},\cdots,f_{n}):(X_{1},\cdots,X_{n})\rightarrow(X^{'}_{1},\cdots,X^{'}_{n})$,
where $f_{i}:X_{i}\rightarrow X^{'}_{i}$.~Moreover,~the composition is given by the rule
$(f_{1},\cdots,f_{n})(g_{1},\cdots,g_{n})=(f_{1}g_{1},\cdots,f_{n}g_{n})$,
where $g_{i}:X^{''}_{i}\rightarrow X_{i}$.~Note the sequence
$$0\rightarrow(A_{1},\cdots,A_{n})\rightarrow(B_{1},\cdots,B_{n})
\rightarrow(C_{1},\cdots,C_{n})\rightarrow0$$
is exact in $\mathcal{X}_{1}\times\cdots\times\mathcal{X}_{n} $~if and only if the sequence
$0\rightarrow A_{i}\rightarrow B_{i}
\rightarrow C_{i}\rightarrow0$
is exact in $\mathcal{X}_{i}$ for any $1\leq i\leq n$.

\begin{lem}\label{lemma-exa-5.6}
Let $\mathcal{X}_{1},\cdots,\mathcal{X}_{n}$ be subcategories of $R\text{-}\mathrm{mod}$, and let $n\geq1$.~Then\\
$(1)$~$[(N_{1},\cdots,N_{n})]_{t}=
[N_{1}]_{t}\times\cdots\times[N_{n}]_{t}$~holds for any $N_{i}\in\mathcal{X}_{i}$ and $t\geq0$,~where $1\leq i\leq n$;\\
$(2)$~$\mathrm{ext.dim}(\mathcal{X}_{1}\times\cdots\times\mathcal{X}_{n})=
\mathrm{max}\{\mathrm{ext.dim}(\mathcal{X}_{i})\mid 1\leq i\leq n\}$.
\end{lem}
\begin{proof}
(1)~The proof proceeds by induction on $t$. For $t=1$.~Let~$(X_{1},\cdots,X_{n})
\in[(N_{1},\cdots,N_{n})]_{1}$.~Then there exists~$(X^{'}_{1},\cdots,X^{'}_{n})\in
R\text{-}\mathrm{mod}\times\cdots\times R\text{-}\mathrm{mod}$ such that
$$(X_{1},\cdots,X_{n})\oplus(X^{'}_{1},\cdots,X^{'}_{n})\cong
(N^{(k)}_{1},\cdots,N^{(k)}_{n})$$
for some integer $k$.~Therefore,~$X_{i}\in[N_{i}]_{1}$, and
$$(X_{1},\cdots,X_{n})\in[N_{1}]_{1}\times\cdots\times[N_{n}]_{1}.$$
Conversely,~let~$(X_{1},\cdots,X_{n})\in[N_{1}]_{1}\times\cdots\times
[N_{n}]_{1}$.~Then~$X_{i}\in
[N_{i}]_{1}$~for~$1\leq i\leq n$.~Hence there exist~$X^{'}_{i}\in R\text{-}\mathrm{mod}$ such that
$$(X_{1},\cdots,X_{n})\oplus(X^{'}_{1},\cdots,X^{'}_{n})
\cong(N^{(l_{1})}_{1},\cdots,N^{(l_{n})}_{n})$$
for some integers $l_{1},\cdots,l_{n}$.~Hence,~$(X_{1},\cdots,X_{n})
\in[(N_{1},\cdots,N_{n})]_{1}$.

Now assume~$t\geq1$~and $(X_{1},\cdots,X_{n})\in[(N_{1},\cdots,N_{n})]_{t}$.~Then, by inductive hypothesis, there exists the following exact sequence
\begin{align}\label{eqno:product-ring}
0\rightarrow(A_{1},\cdots,A_{n})\rightarrow
(X_{1},\cdots,X_{n})\oplus(X^{'}_{1},\cdots,X^{'}_{n})\rightarrow
(B_{1},\cdots,B_{n})\rightarrow0,
\end{align}
with $A_{i}\in[N_{i}]_{t-1},~B_{i}\in[N_{i}]_{1}$.~Hence~$(X_{1},\cdots,X_{n})\in
[N_{1}]_{t}\times\cdots\times[N_{n}]_{t}$.~Conversely,~let~$(X_{1},\cdots,X_{n})\in
[N_{1}]_{t}\times\cdots\times[N_{n}]_{t}$.~Then there exists the exact sequence
$$0\rightarrow A_{i}\rightarrow X_{i}\oplus X^{'}_{i}\rightarrow B_{i}
\rightarrow0$$
such that $A_{i}\in[N_{i}]_{t-1},~B_{i}\in[N_{i}]_{1}$~for $1\leq i\leq n$.~Hence we re-obtain the exact sequence (\ref{eqno:product-ring}),~and~by inductive hypothesis we obtain that
\begin{align*}
&(A_{1},\cdots,A_{n})\in[N_{1}]_{t-1}\times\cdots\times[N_{n}]_{t-1}
=[(N_{1},\cdots,N_{n})]_{t-1},\\
&(B_{1},\cdots,B_{n})\in[N_{1}]_{1}\times\cdots\times[N_{n}]_{1}
=[(N_{1},\cdots,N_{n})]_{1}.
\end{align*}
Hence~$(X_{1},\cdots,X_{n})\in[(N_{1},\cdots,N_{n})]_{t}$.

(2)~Assume~$a:=\mathrm{ext.dim}(\mathcal{X}_{1}\times\cdots\times\mathcal{X}_{n})$.~Then there exist~$N_{1},\cdots,N_{n}\in R\text{-}\mathrm{mod}$ such that
$$\mathcal{X}_{1}\times\cdots\times\mathcal{X}_{n}\subseteq[(N_{1},\cdots,N_{n})]_{a+1}
=[N_{1}]_{a+1}\times\cdots\times
[N_{n}]_{a+1}.$$
Thus $\mathcal{X}_{i}\subseteq[N_{i}]_{a+1}$, and~$\mathrm{ext.dim}(\mathcal{X}_{i})
\leq a$ for all $1\leq i\leq n$.~Next,~assume~$b:=
\mathrm{max}\{\mathrm{ext.dim}(\mathcal{X}_{i})\mid 1\leq i
\leq n\}$.~Then there exist $M_{i}\in R\text{-}\mathrm{mod}$ such that
$\mathcal{X}_{i}\subseteq[M_{i}]_{b+1}$.~Let~$(X_{1},\cdots,X_{n})
\in\mathcal{X}_{1}\times\cdots\times\mathcal{X}_{n}$.~Then, by (1),~we have
$$(X_{1},\cdots,X_{n})\in[M_{1}]_{b+1}\times\cdots\times
[M_{n}]_{b+1}=[(M_{1},\cdots,M_{n})]_{b+1}.$$
Hence~$\mathrm{ext.dim}(\mathcal{X}_{1}\times\cdots\times\mathcal{X}_{n})\leq b$.
\end{proof}

Now, we apply Theorem \ref{thm-clft3.10} in the following result.

\begin{cor}
{\rm Let~$\Lambda:=\begin{psmallmatrix}
R & _{R}N_{S} \\
_{S}M_{R} & S
\end{psmallmatrix}$ be a Morita context ring which is an Artin algebra.~If $N\otimes_{S}M=M\otimes_{R}N=0$,~$_{R}N_{S}$ and $_{S}M_{R}$ are projective bimodules,~then\\
$(1)$~$\mathrm{ext.dim}(\Omega^{j}((R\times S)\text{-}\mathrm{mod}))\leq
\mathrm{ext.dim}(\Omega^{j}(\Lambda\text{-}\mathrm{mod}))$;\\
$(2)$~$\mathrm{ext.dim}(\Omega^{j}(\Lambda\text{-}\mathrm{mod}))\leq
2\mathrm{max}\{\mathrm{ext.dim}(\Omega^{j}(R\text{-}\mathrm{mod})),
\mathrm{ext.dim}(\Omega^{j}(S\text{-}\mathrm{mod}))\}+1$.
}
\end{cor}
\begin{proof}
It is well-known that $\Lambda\text{-}\mathrm{mod}$
is equivalent to a category whose objects are tuples $(X,Y,f,g)$,~where $X\in R\text{-}\mathrm{mod}$,~$Y\in S\text{-}\mathrm{mod}$,~$f:X\otimes_{R}N\rightarrow Y$,~$g:Y\otimes_{S}M\rightarrow X$.~Since~$\Lambda$~is an Artin algebra,~by \cite[Proposition~7.5]{Beligiannis2000} and~\cite[Remark~2.9]{MZL2026}, there is a cleft extension
$$\xymatrix@R=1.8cm@C=1.8cm{
(R\times S)\text{-}\mathrm{mod}
    \ar^-{\mathrm{i}}[r]
& \Lambda\text{-}\mathrm{mod}
    \ar^-{\mathrm{e}}[r]
& (R\times S)\text{-}\mathrm{mod}
    \ar_-{\mathrm{l}}@/_1.8pc/[l] },$$
and functors $\mathrm{e}$,~$\mathrm{l}$ and $\mathrm{F}$ satisfy the following condition:
\begin{align*} 
&\mathrm{e}(X,Y,f,g)=(X,Y);\\
&\mathrm{l}(X,Y)=(X,M\otimes_{R}X,0,1)\oplus(N\otimes_{S}Y,Y,1,0);\\
&\mathrm{F}(X,Y)=(N\otimes_{S}Y,M\otimes_{R}X).
\end{align*}
Obviously,~$\mathrm{e}$ preserves projectives and~$N\otimes_{S}M=M\otimes_{R}N=0$~implies that~$\mathrm{F}^{2}=0$.~Moreover,~the one-side projective properties of~$_{R}N_{S}$ and $_{S}M_{R}$ imply that~$\mathrm{l}$ is an exact functor.~Hence, by Theorem \ref{thm-clft3.10}, we obtain
\begin{align*}
\mathrm{ext.dim}(\Omega^{j}((R\times S)\text{-}\mathrm{mod}))&\leq
\mathrm{ext.dim}(\Omega^{j}(\Lambda\text{-}\mathrm{mod}))\\
&\leq2\mathrm{ext.dim}(\Omega^{j}((R\times S)\text{-}\mathrm{mod}))+1\\
&=2\mathrm{max}\{\mathrm{ext.dim}(\Omega^{j}(R\text{-}\mathrm{mod})),
\mathrm{ext.dim}(\Omega^{j}(S\text{-}\mathrm{mod}))\}+1,
\end{align*}
the last equality follows by Lemma~\ref{lem:syzygy-equivalence} and \ref{lemma-exa-5.6}.
\end{proof}

\section{Local-global properties for extension dimensions}

In this section, we always assume that~$R$ is a commutative Artin ring.~We will study the behavior of extension dimensions under localization.

\begin{lem}\label{Artin-4.1}
{\rm(\cite[Proposition~6.3.4,~Exercise~7.18]{Singh2011})}
Let~$R$ be an Artin ring.~Then $\mathrm{Max}(R)$ is finite and
$R\cong \prod^{n}_{i=1}R_{\mathfrak{m}_{i}}$~as rings,~where~$\mathfrak{m}_{1},\cdots,\mathfrak{m}_{n}$~are all maximal ideals of $R$.
\end{lem}

Next,~we present a lemma needed for the local-global properties of extension dimension.

\begin{lem}\label{lem-local-4.1}
Let~$M, N\in R\text{-}\mathrm{mod}$ and let $n \in \mathbb{N}$. Then the following statements are equivalent:\\
$(1)$~\( N \in [M]_n \);\\
$(2)$~$N_{\mathfrak{p}} \in [M_{\mathfrak{p}}]_n$ for every $\mathfrak{p} \in \mathrm{Spec}(R)$;\\
$(3)$~$N_{\mathfrak{m}} \in [M_{\mathfrak{m}}]_n$ for every $\mathfrak{m} \in \mathrm{Max}(R)$.
\end{lem}
\begin{proof}
From the exactness of~$R_{\mathfrak{p}}\otimes_{R}-$ together with Lemma~\ref{lem:2.4}, we have that $(1)\Rightarrow(2)\Rightarrow(3)$.

$(3)\Rightarrow(1)$ Assume~$N_{\mathfrak{m}} \in [M_{\mathfrak{m}}]_{n}$ for every~$\mathfrak{m}\in\mathrm{Max}(R)$. The proof proceeds by induction on
$n$.~Since~$R$ is an Artin ring,~one has~$R\cong\bigoplus^{a}_{i=1}R_{\mathfrak{m}_{i}}$~by Lemma~\ref{Artin-4.1},~where $\mathfrak{m}_{1},\cdots,\mathfrak{m}_{a}$ are all maximal ideals of $R$.~For~$n=1$, there exist~$X^{\mathfrak{m}_{i}}\in R_{\mathfrak{m}_{i}}\text{-}\mathrm{mod}$ and~$t_{i}\in\mathbb{N}$ such that
\begin{align}\label{Eq-Lem4-1}
N_{\mathfrak{m}_{i}} \oplus X^{\mathfrak{m}_{i}}\cong (M_{\mathfrak{m}_{i}})^{(t_{i})}
 \cong (M^{(t_{i})})_{\mathfrak{m}_{i}}
\end{align}
holds for $1\leq i\leq a$.~Set~$t:=\mathrm{max}\{t_{i}\mid1\leq i\leq a\}$.~Then we obtain an exact sequence in~$R\text{-}\mathrm{mod}$
$$0\rightarrow (\oplus^{a}_{i=1}N_{\mathfrak{m}_{i}})
\oplus(\oplus^{a}_{i=1}X^{\mathfrak{m}_{i}})\oplus
(\oplus^{a}_{i=1}M^{(t-t_{i})}_{\mathfrak{m}_{i}})
\rightarrow \oplus^{a}_{i=1}M^{(t)}_{\mathfrak{m}_{i}}\rightarrow0.$$
Since~$\bigoplus^{a}_{i=1}N_{\mathfrak{m}_{i}}\cong N,~
\bigoplus^{a}_{i=1}M_{\mathfrak{m}_{i}}\cong M$.~Hence~$N\in[M]_{1}$.

Suppose the statement holds for $n-1$ and suppose $N_{\mathfrak{m}_{i}} \in [M_{\mathfrak{m}_{i}}]_n$ for $1\leq i\leq a$, where $\mathfrak{m}_{1},\cdots,
\mathfrak{m}_{a}$ are all maximal ideals of $R$. Then there exists an exact sequence in $R_{\mathfrak{m}_{i}}\text{-}\mathrm{mod}$
\begin{align}\label{eq-local-ex-dim5.1}
0 \rightarrow A^{\mathfrak{m}_{i}}\rightarrow N_{\mathfrak{m}_{i}}\oplus Z^{\mathfrak{m}_{i}}\rightarrow B^{\mathfrak{m}_{i}} \rightarrow 0,
\end{align}
such that $A^{\mathfrak{m}_{i}}\in
[M_{\mathfrak{m}_{i}}]_{n-1}$ and $B^{\mathfrak{m}_{i}}
\in [M_{\mathfrak{m}_{i}}]_1$.
From~(\ref{eq-local-ex-dim5.1}),~we obtain an exact sequence
\begin{align}\label{Eq5.3}
0\rightarrow \oplus^{a}_{i=1}A^{\mathfrak{m}_{i}}
\rightarrow N\oplus(\bigoplus^{a}_{i=1}Z^{\mathfrak{m}_{i}})\rightarrow
\oplus^{a}_{i=1}B^{\mathfrak{m}_{i}} \rightarrow 0
\end{align}
in $R\text{-}\mathrm{mod}$,~and satisfying
\begin{align*}
&\oplus^{a}_{i=1}A^{\mathfrak{m}_{i}}
\in[\oplus^{a}_{i=1}M^{\mathfrak{m}_{i}}]_{n-1}=[M]_{n-1},\\
&\oplus^{a}_{i=1}B^{\mathfrak{m}_{i}}
\in[\oplus{^{a}_{i=1}}M^{\mathfrak{m}_{i}}]_{1}=[M]_{1}.
\end{align*}
Hence~$N\in[M]_{n}$ by (\ref{Eq5.3}).
\end{proof}

\begin{thm}\label{th:4.2}
For any $j,n\in\mathbb{N}$,~the following statements~are equivalent:\\
$(1)$~$\mathrm{ext.dim}(\Omega^{j}(R\text{-}\mathrm{mod}))\leq n$;\\
$(2)$~$\mathrm{ext.dim}(\Omega^{j}(R_{\mathfrak{p}}\text{-}\mathrm{mod}))\leq n$ for every $\mathfrak{p}\in\mathrm{Spec}(R)$;\\
$(3)$~$\mathrm{ext.dim}(\Omega^{j}(R_{\mathfrak{m}}\text{-}\mathrm{mod}))\leq n$ for every $\mathfrak{m}\in\mathrm{Max}(R)$.
\end{thm}
\begin{proof}
$(1)\Rightarrow(2)$ Assume~$\mathrm{ext.dim}(\Omega^{j}(R\text{-}\mathrm{mod}))\leq n$.~Then there exists~$M\in R\text{-}\mathrm{mod}$ such that
$$a:=\mathrm{inf}\{t\geq0\mid\Omega^{j}(R\text{-}\mathrm{mod})\subseteq[M]_{t+1}\}\leq n.$$
Since $R_{\mathfrak{p}}\otimes_{R}-$ is an exact functor for every $\mathfrak{p}\in\mathrm{Spec}(R)$,~by Lemma \ref{lem:2.4},
$$R_{\mathfrak{p}}\otimes_{R}\Omega^{j}(R\text{-}\mathrm{mod})\subseteq R_{\mathfrak{p}}\otimes_{R}([M]_{a+1})\subseteq[M_{\mathfrak{p}}]_{a+1}.$$
Set $b:=\mathrm{inf}\{t\geq0\mid
R_{\mathfrak{p}}\otimes_{R}\Omega^{j}(R\text{-}\mathrm{mod})
\subseteq[M_{\mathfrak{p}}]_{t+1}\}$.~Obviously,~$b\leq a$.~Moreover,~by Lemma~\ref{lem3.1},~we obtain
\begin{align}\label{Eq-Lem4-8}
R_{\mathfrak{p}}\otimes_{R}\Omega^{j}(R\text{-}\mathrm{mod})
\subseteq\Omega^{j}(R_{\mathfrak{p}}\text{-}\mathrm{mod}).
\end{align}
Set~$c:=
\mathrm{inf}\{t\geq0\mid\Omega^{j}(R_{\mathfrak{p}}\text{-}\mathrm{mod})
\subseteq[M_{\mathfrak{p}}]_{t+1}\}$.~Then,~by~(\ref{Eq-Lem4-8}), one has $b\leq c$.~Suppose $b<c$.~Then there exists $Y\in
\Omega^{j}(R_{\mathfrak{p}}\text{-}\mathrm{mod})$ such that
$$R_{\mathfrak{p}}\otimes_{R}Y\cong Y\notin[M_{\mathfrak{p}}]_{b+1}.$$
Moreover,~it follows from $R$ is an Artin ring that all prime ideals are maximal, and by Lemma~\ref{Artin-4.1} we know $R_{\mathfrak{p}}$ is a direct summand of $R$ for any $\mathfrak{p}\in\mathrm{Spec}(R)$.~This implies that $R_{\mathfrak{p}}$ is a finitely generated and projective $R$-module, thus $_{R}Y\in
\Omega^{j}(R\text{-}\mathrm{mod})$ by Lemma~\ref{lem3.1},
this contradicts the definition of $b$.~Hence,~$c=b\leq a\leq n$,~and $\mathrm{ext.dim}
(\Omega^{j}(R_{\mathfrak{p}}\text{-}\mathrm{mod}))\leq n$ for any $\mathfrak{p}\in\mathrm{Spec}(R)$.

$(2)\Rightarrow(3)$ Obviously.

$(3)\Rightarrow(1)$~Assume~$t_{\mathfrak{m}}:=
\mathrm{ext.dim}(\Omega^{j}(R_{\mathfrak{m}}\text{-}\mathrm{mod}))\leq n$ for any $\mathfrak{m}\in\mathrm{Max}(R)$.~Then there exists $M^{\mathfrak{m}}\in
R_{\mathfrak{m}}\text{-}\mathrm{mod}$ such that
$$t_{\mathfrak{m}}=\mathrm{inf}\{i\geq0\mid\Omega^{j}(R_{\mathfrak{m}}\text{-}\mathrm{mod})
\subseteq[M^{\mathfrak{m}}]_{i+1}\}.$$
Since $R_{\mathfrak{m}}\otimes_{R}-$ is an exact functor and preserves projectives,~$R_{\mathfrak{m}}\otimes_{R}\Omega^{j}(R\text{-}\mathrm{mod})
\subseteq[M^{\mathfrak{m}}]_{t+1}$~by Lemma~\ref{lem3.1},
where $t=\mathrm{max}\{t_{\mathfrak{m}}\mid \mathfrak{m}\in
\mathrm{Max}(R)\}$.~Note
\begin{align*}
R_{\mathfrak{m}}\otimes_{R}M^{\mathfrak{n}}
\cong R_{\mathfrak{m}}\otimes_{R}(R_{\mathfrak{n}}
\otimes_{R}M^{\mathfrak{n}})
\cong\begin{cases}
M^{\mathfrak{m}}, &\mathfrak{n}=\mathfrak{m},\\[2mm]
0, &\mathfrak{n}\neq\mathfrak{m},
\end{cases}
\end{align*}
hold for any $\mathfrak{m},\mathfrak{n}\in\mathrm{Max}(R)$.~Therefore
$$M^{\mathfrak{m}}\cong R_{\mathfrak{m}}\otimes_{R}M^{\mathfrak{m}}
\cong R_{\mathfrak{m}}\otimes_{R}(\bigoplus^{a}_{i=1}M^{\mathfrak{m}_{i}})$$
for some $M^{\mathfrak{m}_{i}}\in R_{\mathfrak{m}_{i}}\text{-}\mathrm{mod}$,
where $\mathfrak{m}_{1},\cdots,\mathfrak{m}_{a}$ are all maximal ideals of $R$.~Hence,
$$R_{\mathfrak{m}}\otimes_{R}\Omega^{j}(R\text{-}\mathrm{mod})\subseteq
[R_{\mathfrak{m}}\otimes_{R}(\bigoplus^{a}_{i=1}M^{\mathfrak{m}_{i}})]_{t+1}$$
for any $\mathfrak{m}$.~And then Lemma~\ref{lem-local-4.1} yields $\Omega^{j}(R\text{-}\mathrm{mod})
\subseteq[\bigoplus^{a}_{i=1}M^{\mathfrak{m}_{i}}]_{t+1}$.~Thus
$\mathrm{ext.dim}(\Omega^{j}(R\text{-}\mathrm{mod}))\leq t.$
\end{proof}

\begin{cor}\label{cor4.5}
{\rm It follows~by Theorem~\ref{th:4.2} and Corollary~\ref{prop:3.6} that~$R$ is syzygy finite if and only if $R_{\mathfrak{p}}$ is syzygy finite for all prime ideals~$\mathfrak{p}$.
}
\end{cor}

Suppose that $R$ and $S$ are commutative Artin rings. It is well known that the
ideals of $R\times S$ have the form $I\times J$ where $I$ is an ideal of $R$ and $J$ is an ideal of $S$.~Moreover,~it follows by \cite{AK2008} that the prime (resp. maximal) ideals of~$R\times S$ have the form $\mathfrak{p}\times S$ or~$R\times
\mathfrak{q}$ where $\mathfrak{p}$ is a prime (resp. maximal) ideal of $R$ or $\mathfrak{q}$ is a prime (resp. maximal) ideal of $S$.~Next,~we recall a lemma of localization of $R\times S$,~which is need for Example~\ref{Exa-local-5.8}.

\begin{lem}\label{lem:product-localization}
{\rm(\cite[Exercises~2.32]{Singh2011})}~Let $A$ and $B$ be commutative rings. Then, for every ideals~$\mathfrak p\in\operatorname{Spec}(A)$ and
$\mathfrak q\in\operatorname{Spec}(B)$, one has
\[
(A\times B)_{\mathfrak p\times B}\cong A_{\mathfrak p},~(A\times B)_{A\times\mathfrak q}\cong B_{\mathfrak q}.
\]
\end{lem}

\begin{exa}\label{Exa-local-5.8}
{\rm Let $k$ be a field, and it is known that $R_{1}:=k[x]/(x^{2})$ and $R_{2}:=k[x,y]/(x^{2},y^{2},xy)$ are local Artin $k$-algebras.~Set $R:=\prod^{2}_{i=1}R_{i}$,~$\mathfrak{m}_{i}\in\mathrm{Max}(R_{i})$.~Then, $R$ has only two maximal ideals~$\mathfrak{M}_{1}$ and~$\mathfrak{M}_{2}$, and it follows by~Lemma~\ref{lem:product-localization} that $R_{\mathfrak{M}_{1}}\cong R_{1},~R_{\mathfrak{M}_{2}}\cong R_{2}$.~Moreover,~by Lemma~\ref{lem:syzygy-equivalence},~there is an equivalence of categories~$\Omega^{j}(R\text{-}\mathrm{mod})\simeq\Omega^{j}(R_{1}\text{-}\mathrm{mod})\times \Omega^{j}(R_{2}\text{-}\mathrm{mod})$.
And by Lemma \ref{lemma-exa-5.6},~we obtain
$$\mathrm{ext.dim}(\Omega^{j}(R\text{-}\mathrm{mod}))=\mathrm{max}\{
\mathrm{ext.dim}(\Omega^{j}(R_{1}\text{-}\mathrm{mod})),
\mathrm{ext.dim}(\Omega^{j}(R_{2}\text{-}\mathrm{mod}))\}.$$
Note that~$\mathrm{ext.dim}(R_{1})=0$ since $R_{1}$ is representation-finite type.~Hence,~if $j=0$,~then
$$\mathrm{ext.dim}(R)=\mathrm{max}\{
\mathrm{ext.dim}(R_{1}),\mathrm{ext.dim}(R_{2})\}\leq\ell\ell(R_{2})-1.$$
}
\end{exa}

\bigskip {\bf Acknowledgement.}
This work was partially supported by the National Natural Science Foundation of China (Grant No. 11261050).

\bigskip {\bf Conflict of interest}
The authors declare there is no conflicts of interest.

\def\Up{\mathrm{U}}
\def\Down{\mathrm{D}}
\def\e{\varepsilon}
\def\modcat{\mathrm{mod}\text{-}}
\def\I{\mathrm{I}}
\def\II{\mathrm{II}}
\def\rad{\mathrm{rad}}
\def\soc{\mathrm{soc}}
\def\lineardim{\mathrm{dim}_{k}}

\bigskip

\noindent\textbf{Pei Luo} \\
College of Mathematics and Statistics, Northwest Normal University, Lanzhou 730070, P. R. China \\
E-mail: \textsf{lp2021978572@163.com}\\

\noindent\textbf{Zhongkui Liu} \\
College of Mathematics and Statistics, Northwest Normal University, Lanzhou 730070, P. R. China \\
E-mail: \textsf{liuzk@nwnu.edu.cn}

\end{document}